\documentclass[12pt]{amsart}
\usepackage{amsmath,amscd}
\usepackage{multirow}
\usepackage[centering,text={15.5cm,22cm},
		marginparwidth=20mm]{geometry}

\newtheorem{thm}{Theorem}[section]

\newtheorem{lemma}[thm]{Lemma}

\theoremstyle{definition}
\newtheorem{defn}[thm]{Definition}
\theoremstyle{remark}
\newtheorem{rem}[thm]{Remark}
\numberwithin{equation}{section}
\newcommand{\To}{\rightarrow}
\newcommand{\A}{\mathbb{A}}
\newcommand{\C} {\mathbb{C}}
\newcommand{\CS} {\mathbb{C}^{*}}
\newcommand{\PP} {\mathbb{P}}
\newcommand{\PO} {\mathbb{P}^{1}}

\newcommand{\w} {\omega}
\newcommand{\EExt}{\operatorname{\mathcal{E}xt}}

\begin{document}

\title[]{Multiple Cover Calculation for the Unramified Compactification of the Moduli Space of Stable Maps}%
\author{Iman Setayesh}%
\address{School of Mathematics, Institute for Research in Fundamental Science (IPM), \newline
P. O. Box 19395-5746, Tehran, Iran}%
\email{setayesh@ipm.ir}

\thanks{}%

%\date{}%
%\dedicatory{}%
%\commby{}%
% ----------------------------------------------------------------
\begin{abstract}
In this paper we compute the multiple cover Gromov-Witten integrals (analog of the Aspinwall-Morrison formula) for the unramified compactification of the moduli space of stable maps to an embedded $\PO$ in a Calabi-Yau threefold $X$ with the normal bundle $\mathcal{O}_{\PO}(-1) \bigoplus \mathcal{O}_{\PO}(-1)$.
\end{abstract}
\maketitle
% ----------------------------------------------------------------
\tableofcontents

\section{Introduction}

Let $X$ be a non-singular variety, we want to count the number of curves in $X$ that lie in the homology class $\beta\in H_2(X,\mathbb{Z})$. For a survey of different curve count strategies see \cite{PT}. 

In \cite{KKO} they construct a new compactification for the space of maps from a smooth $n$-pointed genus $g$ curve to a non-singular target. The idea is that the map is unramified except possibly at the marked points. One of the advantages of this compactification is that we will not get any collapsed components, so some of the difficulties caused by these components in the Gromov-Witten theory of the stable maps do not arise. For definition and basic properties of stable maps, see \cite{FP}. 

To be more precise, a map $f$ from a source curve $C$ to a nonsingular target $X$ is unramified at the smooth point $p\in C$ if the differntial map is injective, i.e. the map $df:T_{C,p}\To T_{X,f(p)}$ is injective. In order to compactify the space of maps, they allow the target to vary in the stack of Fulton-MacPherson degenerations of the space $X$. See \cite{FM} for definition and basic properties of Fulton-MacPherson configuration spaces. 

In \cite{KKO} they showed that the moduli space of stable unramified maps, has a perfect obstruction theory. In this paper we compute the genus $0$, multiple cover contribution for an embedded $\PO$ in a Calabi-Yau threefold $X$ with the the normal bundle $\mathcal{O}_{\PO}(-1) \bigoplus \mathcal{O}_{\PO}(-1)$. In the last section we observe that the invariants are not the same as in the Aspinwall-Morrison formula for multiple covers.

The main idea of the computation is to use the virtual localization formula (developed in \cite{GP}). Since the maps in this setup are relative maps, we use the relative virtual localization (see \cite{Li2,GV}). 

The paper is organized as follow: In section 2 we review the definition and basic properties of the stable unramified maps. In section 3 we review the relative virtual localization and add the required changes for our computation. In section 4 we consider a suitable $\CS$-action and describe the fixed point loci. Since the bubbles for any 3-manifold are the same (as they are $\PP^3$ blown up at a number of point), so the description of the fixed point loci works for arbitrary 3-manifolds with similar $\CS$-action. In section 5 we compute the contribution of each fixed point locus to the relative virtual localization formula, and compute the integral for $d$ equal to 2, and by a computer program provide the values for $d < 10$.

{\bf{Acknowledgement.}} The author would like to thank E. Eftekhary, B. Kim and V.Shende for many valuable discussions. A special thanks is due to Rahul Pandharipande for several helpful discussions. The author was partially supported by a grant from Iran's National Elites Foundation.\\

%%%%%%%%%%%%%%%%%%%%%%%%%%%%%%%%%%%%%%%%%%%%%%%%%%%%%%%%%%%%%%%%%%%%%%%%%%%%%%%%%%%%%%%%%%%%%%%%%%%%%%%%%%%%%%%%%%%%%%%%%%%%%%%%%%%%%%%%%%%%%%%%%%%%
\section{Preliminaries}

Let $X$ be a smooth variety. The Fulton-MacPherson configuration space $X[n]$ is a compactification for the space of $n$ distinct marked point on $X$. The idea in this compactification is whenever two points come together we blow up the target to separate them. More precisely, the degenerations of the target can be describe as follow:

Consider the family $X \times \A^1 \to \A$. We blow it up at a sequence of points $\{ p_i \}_{i\geq 0}$ such that each $p_i$ is on the fiber above $0\in\A^1$ and after the first $i$ blow ups $p_i$ is in the smooth locus of the total space. The total space of the fiber over $0$ is a possible Fulton-MacPherson degeneration of $X$. We denote the universal family over $X[n]$ by $X[n]^+$. For more details see\cite{FM}

\begin{defn}

A pair of maps $(\pi_{W/S}$ and $\pi_{W/X})$ 

$$\begin{CD}
W @>\pi_{W/X}>>  X\\
@V\pi_{W/S}VV \\
S 
\end{CD}$$
 from an algebraic space $W$ is called a Fulton-MacPherson degeneration of $X$ over a scheme $S$, if after an \'etale and  surjective base change $T \To S$ there exist an integer $n>0$ and a commutative diagram 
$$\begin{CD}
W_T @>>> X[n]^+ @>>> X\\
@VVV @VVV\\
T @>>> X[n]
\end{CD}$$
 such that after the base change the map $\pi_{W/X}$  is equal to the composition of $W_T \To X[n]^+\To X$.

\end{defn}

In \cite[section 2]{KKO}  they construct an Artin stack which parametrizes the Fulton-MacPherson degenerations of $X$ over a base $S$. We denote this stack by $\mathcal{FM}(X)$, when there is no confusion about the base scheme $S$.

Note that in a Fulton-MacPherson degeneration of $X$ with $dim X=n$ there is a component obtained by blowing up the space $X$ at a number of points, and all the other components are the blow ups of $\PP^n$ . There are two special cases, one is when the component is just $\PP^n$ and the other is the ruled component, that is the blow up of $\PP^n$ at just one point. In \cite{KKO} they allow the map to be ramified in the marked points, but since in our work we do not have any markings, we simplify the definitions for our purpose.

\begin{defn}
Let $X$ be a nonsingular projective variety of dimension $n$. The unramified compactification of moduli space of smooth stable maps from a curve of genus $g$ to $X$, representing the class $\beta\in H_2(X,\mathbb{Z})$, parametrizes the triples $(C,\pi_{W/X} : W \To  X,f:C\to W)$ satisfying:
\begin{enumerate}
\item $C$ is a stable nodal curve with arithmetic genus $g$.
\item $\pi_{W/X} : W \To  X$ is a Fulton-MacPherson degeneration over $k$.
\item $(\pi_{W/X}\circ f)_*[C]=\beta\in H_2(X,\mathbb{Z})$.
\item The inverse image of the nonsingular locus of $W$ is exactly the nonsingular locus of $C$, i.e. $f^{-1}W^{ns}=C^{ns}$.
\item (Admissibility condition) At each node $n\in C$, which by condition $(3)$ is mapped to an intersection divisor of $D \subset W$ the two branches of $B_1,B_2 \subset C$ map to two different component of the target intersecting at $D$, and the intersection multiplicities of them with $D$ at the image of $n$ are the same. 
\item (Stability condition)
\begin{itemize}
\item For each ruled component $W_0$ there is a component of $C$ which is mapped by $f$ into $W_0$ and the image is not a fiber of the ruling. 
\item For each $\PP^n$ component $W_1$ there is a component of $C$ which is mapped by $f$ into $W_1$ and the image is not a straight line.
\end{itemize}
\end{enumerate}

\end{defn}

Condition $5$ is standard in the relative Gromov-Witten theory and in \cite{Li1,Li2} is called predeformability, see also \cite{IP,LR}.

In \cite{KKO} they show that this functor is representable by a Deligne-Mumford stack (which we denote by $\overline{\mathcal{M}}^{u}_{g}(X,\beta)$), and has a perfect obstruction theory so the unramified Gromov-Witten invariants can be defined. 

\section{Relative Deformation Theory}

A perfect obstruction theory for the moduli space of relative stable maps is constructed in \cite{Li2,GV}. Here we explain this construction and add required modifications for our problem. Following \cite{GV} we will consider the natural morphism:

$$\Phi:\overline{\mathcal{M}}^{u}_{0}(X,d) \To \mathcal{FM}(X) \times \mathfrak{M}_0$$

\noindent and the relative obstruction theory for this morphism. This means that we have fixed the source and target, so the deformation and obstruction theory of such a map can be understood in term of known cohomology groups.

Let  $X$ be a smooth scheme and $D$ be a smooth divisor we use $\Omega_{X}(logD)$ to denote the sheaf of differentials with logarithmic pole along $D$, and by $T_{X}(logD)$ its dual (for definition and properties of log-differentials see \cite{Saito}).

Let $Y$ be an FM-configuration space and $f : C \to Y$ be a stable map, then we define $f^{\dag}T_{Y}$ to be the quotient of $f^{*}T_{Y}$ by the torsion subsheaf supported at the nodes of $C$. If $C_i$ is an irreducible component of $C$ being mapped to the component $Y_i$ of the target, the the restriction of sections of this sheaf to $C_i$ are just the sections of the  pullback of the sheaf of differentials of $Y_i$ with logarithmic pole along the singular locus.

Let $f:C \To W$ be a stable map to an FM-configuration space, then we have:

$$RelDef(f)=H^{0}(C,f^{\dag}T_{W})$$

\noindent since when we fixed the target, then a predeformable map is a log smooth map from a nodal curve to a scheme with normal crossing singularities, and so their deformation theory is given by the sheaf of log differentials (see \cite{Kato}).

If we denote the number of nodes of $C$ with image in the $i^{th}$ distinguished divisor (intersection of two smooth components) of $W$ by $n_i$, and the line bundle associated to smoothing this divisor by $L_i$, i.e. if $D_i$ is the $i^{th}$ divisor be the intersection of two irreducible components  $W^0$ and $W^1$ of $W$ then  $L_i=N_{D_i}W^0\otimes N_{D_i}W^1$. The space of relative obstructions for this map has a natural filtration:

$$0 \To H^{1}(C,f^{*}T_{W}) \To RelOb(f) \To H^{1}(C,f^{-1}\EExt^{1}(\Omega_{W},\mathcal{O}_{W}))\simeq \bigoplus L_i^{n_i}\To 0$$

\noindent In the above sequence the first term is the obstruction space for deformation of $f$ as a log smooth morphism, and the term on the right is a local obstruction coming from a compatibility condition. (For a discussion about this see \cite{GV} page 9)

In our computation we will use three types of exact sequences, each of them allows us to reduce the computation to a simpler case. We will say a few words about them.

\vspace{1mm}
\begin{enumerate}

  \item Let $\Sigma= \Sigma_{1} \bigcup \Sigma_{2}$ be a  nodal curve with one node, and $Y=Y_1 \bigcup Y_2$ be a scheme with normal crossing singularities along the intersection of $Y_1$ and $Y_2$ which we will denote by $D$, and let $g$ be a map from $\Sigma$ to $Y$ , then we have the analog of the normalization sequence for the sheaf of log differentials:
 \begin{align*}
0 \To  H^{0}(\Sigma,f^{\dag}T_{Y})  \To H^{0}(\Sigma_{1},f^{*}T_{Y_{1}}(-logD))\oplus H^{0}(\Sigma_{2},f^{*}T_{Y_{2}}(-logD)) \To T_{g(p)}D  \\
\To H^{1}(\Sigma,f^{\dag}T_{Y})  \To  H^{1}(\Sigma_{1},f^{*}T_{Y_{1}}(-logD))\oplus H^{1}(\Sigma_{2},f^{*}T_{Y_{2}}(-logD)) \To 0 \\ \tag{$\ast$}
\end{align*}

\noindent this allows us to reduce the computation of deformation/obstruction contributions to the irreducible components. 

 \item We also need to relate the cohomology of the logarithmic tangent bundle to those of the ordinary tangent bundle. Let $Y$ be a smooth scheme and $D$ a smooth effective divisor, we have the following exact sequence of sheaves: \vspace{3mm}

\begin{center}
 $ 0 \To T_{Y}(-logD) \To T _{Y} \To N_{D}Y \To 0$
\end{center} \vspace{3mm}

\noindent in which the first sheaf is the sheaf of tangent vector fields which are tangent to $D$, so a section of the quotient sheaf consists of a normal vector field over $D$. So let $\Sigma $ be a curve and $g$ a map from $\Sigma$ to $Y$, if we pull back the above sequence to $\Sigma$ and take the cohomology we get:
\begin{align*}
0 \To H^{0}(\Sigma,g^{*}T_{Y}(-logD)) \To H^{0}(\Sigma,g^{*}T_{Y}) \To H^{0}(\Sigma,g^{*}N_{D}Y) \\
\To H^{1}(\Sigma,g^{*}T_{Y}(-logD)) \To H^{1}(\Sigma,g^{*}T_{Y}) \To 0
\end{align*}

There is one particular case that we use it in this paper, when the ambient space is $\PP^n$ and $D$ is the hyperplane defined by $x_0=0$. In this case if the analog of Euler's exact sequence for logarithmic differentials shows that:\vspace{3mm}

\begin{center}
     $\begin{array} {lcccr}
     0 \To \mathcal{O} \To & \mathcal{O}\oplus\mathcal{O}(1)^{\oplus n} & \To &  T_{\PP^n}(-logD) & \To 0 \\
     & (f_0,(f_i)_{i\geq 1}) & \mapsto & f_0 x_0 \frac{\partial}{\partial x_0} + \sum\limits _{i\geq 1 } f_i \frac{\partial}{\partial x_i} &
 \end{array}$
 \end{center}\vspace{3mm}

\noindent which gives us $T_{\PP^n}(-logD) = \mathcal{O}(1)\otimes W^{n}$.

\item The third exact sequence that we will need is the one relating the tangent bundle of the blow up to the tangent bundle of the original scheme. Let $Z$ be a scheme and $q$ be a smooth point, and $\tilde{Z}$ be the blow up of $Z$ at $q$, then we have: \vspace{3mm}

\begin{center}
$0 \To T_Z \To p^{*}T_{\tilde{Z}} \To j_{*}\mathcal{Q} \To 0$
\end{center} \vspace{3mm}

\noindent where $p$ is the projection from $\tilde{Z}$ to $Z$, and $j$ is the inclusion of the exceptional divisor, and $\mathcal{Q}$ is the universal quotient bundle over the exceptional divisor, which means that on the exceptional divisor it's given by:\vspace{3mm}

\begin{center}
$0 \To \mathcal{O}(-1) \To \mathcal{O}\otimes V^{3} \To \mathcal{Q} \To 0$
\end{center} \vspace{3mm}
\end{enumerate}

%%%%%%%%%%%%%%%%%%%%%%%%%%%%%%%%%%%%%%%%%%%%%%%%%%%%%%%%%%%%%%%%%%%%%%%%%%%%%%%%%%%%%%%%%%%%%%%%%%%%%%%%%%%%%%%%%%%%%%%%%%%%%%%%%%%%%%%%%%%%%%%%%%%%%%%%%%%%%%

\section{Fixed Point Loci}

 We will consider the local model for $X$ to be the total space of the $\mathcal{O}_{\PO}(-1) \bigoplus \mathcal{O}_{\PO}(-1)$, and take the $\CS$-action to be as follow:

In the first copy of $\mathcal{O}(-1)$ which can be realized as $\C^{2} \backslash \{0\}$ we have $t.(x,y)=(x,t^{-1}y)$, and on the second copy it's given by $t.(x,y)=(tx,y)$. It's clear that both of these actions induce the same action on the base $\PO$ so we have a well defined action on $X$ . Let $-\alpha$ be the weight of the standard action of $\CS$ on $\mathbb{C}$. 

For a fixed map the image of the ramification of the map are invariant under the action so the ramification points are above the two fixed points of the action on $\PO$, i.e. $0$ and $\infty$. It means that for the target of such maps we have to consider a FM configuration space for $X$ with bubbles only above $0$ and $\infty$. The weights of the $\CS$-action on the tangent space at both of these points are: $\{\alpha ,0,-\alpha\}$. For example at $0=[0:1]$ we have: 

The weights of the action on the fiber of first $\mathcal{O}(-1)$ is $\alpha$, the fiber of the second copy of $\mathcal{O}(-1)$ over $0$ is $\CS$-fixed, i.e. the weight is zero, and the weight of the action on the tangent space of $\PO$ at $0$ is $-\alpha$. So the set of weights for the bubble above $0$ (and similarly for $\infty$) are $\{\alpha,0,0,-\alpha\}$, the extra 0 weight comes from the copy of trivial bundle in $\PP(T_{X}\oplus I)$.

 We split the target into three components, one is $X$ blown up at two points which we denote by $X_{0,\infty}$, and the other two components are the bubbles above zero and infinity, which we denote by $Y_0$ and $Y_{\infty}$. By relative virtual localization theorem, in order to compute the contribution of each fixed point component, its enough to compute the contribution coming from maps with target in each of these three pieces, and combine them according to the formula.

Now we will give a description of the fixed maps. We start by looking at a component of the source with its image in one of the bubbles. We start by considering the bubble above 0, and denote the coordinate of the $\PP^3$ by $t_{0},t_{1},t_{2},t_{3}$, which are associated to (resp.) $\mathcal{T}_{X}$, first and second copy of $\mathcal{O}(-1)$, and the coordinate associated to the trivial bundle. The induced action is given by $t.(t_{0},t_{1},t_{2},t_{3})=(t^{-1}.t_{0},t.t_{1},t_{2},t_{3})$. Then the distinguished divisor is given by $t_{3}=0$. We denote the component of the source with the image in the first bubble above $0$ by $C_1$. The image of $C_{1}$ should meet the distinguished divisor ($t_3=0$) at the intersection of $\PO$ and exceptional divisor which is $[1,0,0,0]$.

Since we know that image of $f$ meets the $t_{3}=0$ at one point with multiplicity $d$ then the restriction of $f$ to $C_{1}$ should be of the following form:\vspace{3mm}

     $$\begin{array} {ccccl}
     f & : & \PO & \To & \PP^3 \\
      & & [x;y] & \mapsto & [g_{1}(x,y);g_{2}(x,y);g_{3}(x,y);y^d g_{4}(x,y)] \\
       & & & & \text{with } deg(g_{1})=deg(g_{2})=deg(g_{3})=d+deg(g_{4})=n
 \end{array}$$

Since the class of this map is fixed, it means that up to an automorphism of the source and an automorphism of the target, the class of $t.f$ remains unchanged. The automorphism group of target if it has no point blown up is generated by dilation and linear transformations of $\A^3 \subset \PP^3$. Any automorphism of the source curve should fix the intersection points of $C$ and the special divisor i.e. $t_3=0$. So if $g_4$ has at least 2 distinct roots then the automorphism group of source is a subgroup of the finite group of automorphisms of $\PO$ which permutes these points.

Here we consider two different cases:

\begin{enumerate}
  \item $g_4$ has at most one root, i.e. $g_4(x,y)=x^{n-d}$ ($n\geq d$). In this case the intersection point(s) of $C_{1}$ and the special divisor should be fixed.
  \item $g_4$ has at least two distinct roots.
\end{enumerate}

\vspace{3mm}
\textbf{Case 1:} In this case for each $t \in \CS$ there exist $u,v \in \CS$,$a,b,c,w \in \C$ such that:

\noindent $[g_{1}(x,y);g_{2}(x,y);g_{3}(x,y);y^d x^{n-d}] =$
\begin{flushright}
$ [t^{-1}.v.g_{1}(ux+wy,y)+a.y^d (ux+wy)^{n-d};t.v.g_{2}(ux+wy,y)+b.y^d (ux+wy)^{n-d}; v.g_{3}(ux+wy,y)+c.y^d (ux+wy)^{n-d};y^d (ux+wy)^{n-d}]$
\end{flushright}

 \noindent which means for some $\lambda \neq 0$ we have: \vspace{3mm}

 $$\left \{ \begin{array}{lll}
 g_{1}(x,y)&=&\lambda t^{-1}.v.g_{1}(ux+wy,y)+a.y^d g_{4}(x,y) \\
 g_{2}(x,y)&=&\lambda t.v.g_{2}(ux+wy,y)+b.y^d g_{4}(x,y) \\
 g_{3}(x,y)&=&\lambda v.g_{3}(ux+wy,y)+c.y^d g_{4}(x,y) \\
 x^{n-d}&=&\lambda (ux+wy)^{n-d}
 \end{array} \right. $$

We know that $g_{1}(0,1)\neq0$ (since the intersection point with the distinguished divisor is $[1;0;0;0]$) so $g_{1}$ has a summand of $x^n$, by considering the coefficient of this term in both sides of the first equation we get: $v.u^d=t$.

The last identity shows:

$$ \begin{cases}
   \lambda =u^{d-n} , w=0 & \text{if } n \neq d \\
   \lambda =1        & \text{if } n = d
  \end{cases}$$

In both cases up to an automorphism of the target (i.e. an appropriate choice of $a$,$b$ and $c$) we can assume that $g_1$, $g_2$ and $g_3$ have no summand of the form $x^{n-d}y^d$.

\begin{itemize}
  \item If $n \neq d $ by the above assumption we have $a=b=c=0$, so we have:

   $$\left \{ \begin{array}{lll}
 u^{n-d} g_{1}(x,y)&=&t^{-1}.v.g_{1}(ux,y) \\
 u^{n-d} g_{2}(x,y)&=&t.v.g_{2}(ux,y) \\
 u^{n-d} g_{3}(x,y)&=&v.g_{3}(ux,y)
 \end{array} \right. $$

  and if we write $ g_{1}(x,y) = \displaystyle \sum_{i=0}^{n} a_{i}x^{i}y^{d-i} $, $ g_{1}(x,y) = \displaystyle \sum_{i=0}^{h} b_{i}x^{i}y^{d-i} $ and $ g_{1}(x,y) = \displaystyle \sum_{i=0}^{k} c_{i}x^{i}y^{d-i} $ with $a_n\neq 0$,$b_h\neq 0$ and$c_k\neq 0$ then we get:

 $$ \begin{cases}
   a_i (t^{-1} v - u^{n-d-i} )=0 & \text{for } i=0,...,n \\
   b_i (tv-u^{n-d-i})=0        & \text{for } i =0,...,h  \\
   c_i (v-u^{n-d-i})=0        & \text{for } i=0,...,k
  \end{cases}$$

  If either $h$ or $k$ are not zero then by looking at the term of degree $h$ of $g_2$ or degree $k$ of $g_3$ , together with $v.u^d=t$ we can write $u$ and $v$ in terms of $t$ (up to a root of unity). Then the above system shows that all these $g_i$'s should be monomials, and $f$ has one of the following forms:

  \begin{center}
     $\begin{array} {ccccll}
       f & : & \PO & \To & \PP^3 & \\
         & & [x;y] & \mapsto & [ x^n;\alpha x^{h}y^{n-h};\beta x^{k}y^{n-k};x^{n-d}y^d] &  \\
               & & & & & \text{with } n+h=2k \\
          \text{or} & & & & &\\
         & &  [x;y] & \mapsto & [x^n;x^{h}y^{n-h};0;x^{n-d}y^d]&\\
          & & & & & \text{with } 1 \leq h \leq d-1 \\
      \text{or} & & & & &\\
      & &  [x;y] & \mapsto & [x^d;0;x^{k}y^{n-k};x^{n-d}y^d]& \\
      & & & & & \text{with } 1 \leq k \leq d-1 \\
 \end{array}$
 \end{center}\vspace{3mm}

  If $h=k=0$ then $f$ has the following form:

  $$\begin{array} {ccccl}
     f & : & \PO & \To & \PP^3 \\
      & & [x;y] & \mapsto & [g_{1}(x,y);0;0;x^{n-d}y^d] \\
       & & & & \text{with } deg(g_{1})=n
 \end{array}$$

  \item If $n=d$, similarly we have:

By considering the terms with highest $x$ degree on $g_{2}$ and $g_{3}$, if those terms are $x^{h}y^{d-h}$ and $x^{k}y^{d-k}$ (resp.) with $k,h \neq 0$, we get:  $t.v=u^{-h}$ and $v=u^{-k}$, which implies: \vspace{3mm}

  $$\left \{ \begin{array}{lll}
 u^{d-2k+h}&=&1 \\
 u^{d-k}&=&t \\
 u^{-k}&=&v
 \end{array} \right. $$

 \noindent and since this system of equations should have a solution for each $t\in\CS$, we have $d+h=2k$, and since $g_{2}(0,1)=g_{3}(0,1)=0$, then we have $h,k\leq d-1$.

 For $d\geq 2$ if we write $g_{1}$ as $ g_{1}(x,y) = \displaystyle \sum_{i=0}^{d} a_{i}x^{i}y^{d-i} $  since $g_1(x,y)$ satisfies:

 \noindent $g_{1}(x,y)=t^{-1}.v.g_{1}(ux+wy,y)+a.y^d $  \hspace{3mm} (for each $i$), we have:

 $$ a_{i}.u^{d}= \displaystyle \sum_{j=i}^{d}{j \choose i} a_{j}u^{i}w^{j-i}$$

 For $i=d-1$ this gives us

 \begin{equation}\label{a_i eqn}
    w=\displaystyle\frac{a_{d-1}}{da_{d}}(u-1)
 \end{equation}

 \noindent and we see that unless we have $h=k=0$ we can write $u$ in terms of $t$ (up to a root of unity), so we can view (\ref{a_i eqn}) as an identity in term of the parameter $t$. In this case for $1\leq i\leq d-2 $ by looking at the term with highest $u$ power on both sides we get:

$$ a_{i}= a_{d} \displaystyle {d \choose i}  \left(\frac{a_{d-1}}{da_{d}}\right)^{d-i} $$

\noindent and we can see that with these $a_{i}$'s we have: \vspace{3mm}

$$\begin{array}{lll}
  \displaystyle \sum_{j=i}^{d}{j \choose i} a_{j}u^{i}w^{j-i} &=& \displaystyle \sum_{j=i}^{d}{j \choose i}{d \choose j} a_{d} (\frac{a_{d-1}}{da_{d}})^{d-j}u^{i}w^{j-i} \\
 &=& \displaystyle {d \choose i} a_{d} u^{i} \sum_{j=i}^{d}{d-i \choose j-i}  (\frac{a_{d-1}}{da_{d}})^{d-j}w^{j-i}\\
 &=& \displaystyle {d \choose i} a_{d} u^{i} (\frac{a_{d-1}}{da_{d}}u)^{d-i}\\
 &=& a_{i}.u^{d}\\
 \end{array}$$

If we substitute $a_i$'s in $g_1$ we see that modulo a factor of $y^d$ it has the following form:

 $$\begin{array}{lll}
 g_{1}(x,y) &=& \displaystyle \sum_{i=0}^{d} a_{i}x^{i}y^{d-i} \\
 &=& \displaystyle \sum_{i=0}^{d}  {d \choose i} a_{d} (\frac{a_{d-1}}{da_{d}})^{d-i}x^{i}y^{d-i} \\
 &=& a_{d} (x +\displaystyle\frac{a_{d-1}}{da_{d}}y)^{d} \vspace{3mm} \\
 &=& a_{d} (x +\displaystyle\frac{w}{u-1}y)^{d}
 \end{array}$$

  By the same argument we can show that $g_2$ and $g_3$ can be written in a similar form.

 $$\begin{array}{lll}
 g_{2}(x,y) &=& b_{h} (x +\displaystyle\frac{w}{u-1}y)^{h} y^{n-h} \\
 & & \\
 g_{3}(x,y) &=& c_{k} (x +\displaystyle\frac{w}{u-1}y)^{k} y^{n-k}
 \end{array}$$

  So unless $h=k=0$ by a change of parameter on the source curve, we can choose a representative for $f$ in which all $g_i$'s are monomials, and $f$ has a representative of the following form: \vspace{3mm}

\begin{center}
     $\begin{array} {ccccll}
       f & : & \PO & \To & \PP^3 & \\
         & & [x;y] & \mapsto & [ x^n;l x^{h}y^{n-h};m x^{k}y^{n-k};y^n] & \text{with } n+h=2k
 \end{array}$
 \end{center}\vspace{3mm}

 Note that if we have $h=0$ or $k=0$ then by the same argument as above after an appropriate automorphism, $f$ can be written in one of the following forms:

 $$\begin{array} {rccccll}
      & f & : & \PO & \To & \PP^3 &\\
      & & & [x;y] & \mapsto & [x^n;x^{h}y^{n-h};0;y^n]&\\
      \text{or}& & & & & &\\
      & & & [x;y] & \mapsto & [x^n;0;x^{k}y^{n-k};y^n]& \\
      & & & & & & \text{with } 1 \leq h,k \leq n-1 \\
 \end{array}$$

If $h=k=0$ then

$$\begin{array} {ccccl}
     f & : & \PO & \To & \PP^3 \\
      & & [x;y] & \mapsto & [g_{1}(x,y);0;0;y^n] \\
       & & & & \text{with } deg(g_{1})=n
 \end{array}$$

\noindent using the above notation, by picking $v=t$ we see that the map $f$ is fixed under the $\CS$ action.

For $d=1$, the map can be written as $[x;y] \mapsto  [x;\alpha x;\beta x;y]$, and for it to be fixed we should have $\alpha=\beta=0$, i.e. the map is given by:

$$\begin{array} {ccccl}
     f & : & \PO & \To & \PP^3 \\
      & & [x;y] & \mapsto & [x;0;0;y]
 \end{array}$$

\end{itemize}

\textbf{Case 2:} In this case similarly (with notations as above) if have $h \text{ or } k \neq 0$ then we can write $u$ in terms of $t$. But since the automorphism group of the source is finite this is a contradiction, which shows that $g_2=g_3=0$. In this case the $f$ has the following form:

$$\begin{array} {ccccl}
     f & : & \PO & \To & \PP^3 \\
      & & [x;y] & \mapsto & [g_{1}(x,y);0;0;g_4(x,y)] \\
       & & & & \text{with } deg(g_{1})=deg(g_{4})=n
 \end{array}$$

\begin{lemma}
If the generic point of a component of the fixed point loci, has at least two distinct components being mapped to $X_{0,\infty}$ then the contribution of this fixed point locus to $[\overline{\mathcal{M}}_{\Gamma}(X_{0,\infty},D)]^{vir}$ and hence the total contribution to the integral will be zero.
\end{lemma}

 \begin{proof}

 If we denote by ${C_i}$ the components that are being mapped to $X_{0,\infty}$, then we have:

\begin{itemize}
  \item Deformations of the map $f_i$ which (in the K-theory) by the results of last section are given by:
   $$\begin{array}{lll}
\bigoplus\chi(f_i^{*}T_{X_{0,\infty}}(-log D_0 -log D_{\infty}))&=&\bigoplus \chi(f_i^{*}T_X) \\
& & - \left( \bigoplus H^{0}(f_i^{*}\mathcal{O}_{D_{0}})\otimes V_0+\bigoplus H^{0}(f_i^{*}\mathcal{O}_{D_{\infty}})\otimes V_{\infty} \right)
\end{array}$$
  \item $T_{D_0}|_{f(N_i)}$, $T_{D_{\infty}}|_{f({N'}_i)}$, which have no weight zero piece.
  \item For each curve $C_i$ the automorphisms of the source have weighs $\{\frac{\alpha}{d_i},\frac{-\alpha}{d_i},0\}$.
  \item Deformations of the target coming from moving the bubbles have weights $\{\alpha,-\alpha,0\}$ , $\{\alpha,-\alpha,0\}$.
  \item For each curve $C_i$ the deformations of the source coming from moving the special points have weights $\{\frac{\alpha}{d_i},\frac{-\alpha}{d_i}\}$
  \item We also have deformation of the source coming from smoothing the nodes with weights $\{\frac{\alpha}{d_i}+\w_{0},\frac{-\alpha}{d_i}+\w_{\infty}\}$ (where $\w_0$ , $w_\infty$ are the weights of the action on the tangent space of the curves with image in the bubble above $0$ and $\infty$ at respective nodes).
\end{itemize}

We see that the zero weights coming from the automorphisms of the source cancel out with the zero weights coming from deformation of the map $f_i$ (i.e. $\chi(f_i^{*}T_X)$ ). But we also get zero weight contributions from the tangency condition (i.e. $\bigoplus H^{0}(f_i^{*}\mathcal{O}_{D_{\infty}})\otimes V_{\infty}$), since the weights on $V$'s are $\{\alpha,-\alpha,0\}$ and the weights of $H^{0}(f_i^{*}\mathcal{O}_{D_{\infty or 0}})$ are (up to a sign depending on $0$ or $\infty$) $\{\frac{j}{deg(f_i)}\alpha\}_{j=0}^{deg(f_i)}$. In order to get a non-trivial contribution these zero weights should cancel out with the weight zero term coming from deformation of the target, but if we have more than one component they will not cancel out and we will get a zero contribution. This can not cancel out with terms from the bubbles either as we see in the below.

Fix a bubbles $B$, we denote the restriction of $f$ to components with image in $B$ by $g_i$. If $B$ is $\PP^3$ blown up at $m$-points, then we have the following terms:

\begin{itemize}
  \item Deformations of the maps $g_i$ which (in the K-theory) by the results of last section are given by:

   $\bigoplus_{i}\chi(g_i^{*}T_{X_{0,\infty}}(-log D_0 -log D_1-\ldots -log D_m)$
$$=\bigoplus_{i} \chi(g_i^{*}T_X(-log D_0)) - \left( \bigoplus_{i,j} H^{0}(g_i^{*}\mathcal{O}_{D_j})\otimes V_j \right)$$
  \item $T_{D_0}|_{f(N_i)}$, $T_{D_{j}}|_{f({N'}_{i,j})}$
  \item Automorphisms of the target.
  \item For each curve $C_i$ the automorphisms of the source.
  \item Deformations of the target coming from moving the bubbles.
  \item For each curve $C_i$ the deformations of the source coming from moving the special points.
  \item We also have deformation of the source coming from smoothing the nodes.
\end{itemize}

We consider the following three cases separately:

\begin{enumerate}
  \item If the target has at least one point blown up and the image of $C_i$ is a line:

  \textbf{Deformations of the source:} For each node we have a zero weight coming from moving this node, but all these weight zero  terms are in the fixed piece of the normal bundle and have no contribution to the moving part (see \cite[section 3]{GP}). 

  \textbf{Deformations of the map and target:} If the image of $C_i$ passes through $k$ bubbles on the target and has ramification above $k'$ of them, then moving these bubbles gives $k'$ weight zero terms (Note that moving other $k-k'$ bubbles would be taken care of when we consider the contribution of the component passing through those bubbles and ramified over them, and such component exist by the stability condition and the fact that genus of the source curve is zero) Deformation of the map given by $\chi(g_i^{*}\mathcal{O}(1))\otimes V^{3}$ has $d+1$ weight zero terms, which means totally we have $d+k'+1$ weight zero terms from deformation of the map and target.

  \textbf{Automorphisms of the source:} They give us $3$ weight zero terms.

  \textbf{The tangency conditions:} Given by $H^{0}(g_i^{*}\mathcal{O}_{D_j})\otimes V_j$ give us $dk$ zero terms (Since at each node if the ramification index is $p$ we get $p+1$ weight zero terms).

  Hence in order to get a non zero contribution we must have $dk+3 \geq d+k'+1$ which holds since we have:

  $$(d-1)(k-1)+1 > 0 .$$

  \item If the target has at least one point blown up and the image of $C_i$ is not a line (in this case the only ramification happens over $[0:0:0:1]$):

  \textbf{Automorphisms/Deformations of the source:} Since image of $C_i$ is not a line, then by the description of fixed point loci, we know that there is exactly one ramification point on $C_i$. The automorphisms of $C_i$ should fix the two nodes, (we have at least $2$ of them), and the one dimensional remaining automorphisms are given by dilation which gives a zero weight contribution.

  \textbf{Automorphisms/Deformations of the target:} Since the target is blown up in at least one point and the automorphisms should fix that point, the only automorphisms that remain are given by dilation which is a weight zero term.

  \textbf{The tangent space at nodes:} In this case we won't get any zero weights from $T_{D}|_{f(N)}$'s by the above computations.

  \textbf{Deformations of the map given by $\chi(g_i^{*}\mathcal{O}(1))\otimes W^{3}$ :} The $\CS$ action can have $2$ or $3$ weight zero terms.

  If there is a 1-dimensional $\CS$-fixed component passing through this point in the moduli space, then there will be $3$ zero weights and 2 otherwise.

  In the first case 2 of these weights cancel out with the weight zero terms coming from automorphisms of the target and source, and the remaining zero weight is part of the fixed part of the deformation bundle over the 1-dimensional fixed loci so it does not contribute to the normal bundle. Note that if there are more than one such component in one bubble the number of remaining zero weights (those who won't be canceled out with the zeroes from automorphisms of the source and target) is exactly the dimension of the fixed component passing through that point and all these weights are in the fixed part of the deformation and so will not contribute to the normal bundle.

  In the second case both of these weights are canceled out by the terms from automorphisms of the source and target and so have no contribution to the normal bundle computation.

  \item If the target is just $\PP^3$:

  If a map has no ramification inside the $\PP^3$ (away from the special divisor) it can hit the special divisor in at most $2$ points, so the deformations of the source cancel out with the automorphisms of the source.

  The deformations of the map given by $\chi(g_i^{*}\mathcal{O}(1))\otimes W^{3}$, have $2$ or $3$ wight zero terms, and similar argument as above shows that they cancel out with the zero weight terms coming from automorphisms of the source and target and one is corresponding to the 1-dimensional fixed point component passing through this point.

\end{enumerate}

So the weight zero terms coming from the base contribution won't cancel out, hence if we have more than one component being mapped to the base the total contribution will be zero.

\end{proof}

In fact from this computation we see that if we have a component with line image then we get an extra weight zero term, which means even in the case that we have only one component mapped to the base $\PO$ if we have a component with line image in one of the bubbles then the contribution of that fixed locus is also zero.

 So in order to compute the integral we only have to consider the fixed maps with only one component mapping to the base $\PO$, which we denote by $C_{b}$. Let $C_{b}\coprod C^{0}_{i} \coprod C^{\infty}_{j} \To C$ be the normalization of $C$, resolving all the nodes of $C$ with $C^{0}_{i}$(and $C^{\infty}_{j}$) irreducible components above 0 (and $\infty$ resp.).

Since the restriction of $f$ to $C_{b}$ is a degree d map, and it has ramification points only above 0 and $\infty$ it should be fully ramified at both of them, so by the admissibility condition the contact order of $C^{\bullet}_{1}$ with the exceptional divisors is d at the nodes, and these are the only points where  $C^{\bullet}_{1}$ meet these divisors, since if there were any other contact point say $p$ then the image of $p$ lies in the singular part of the target so it should be a singular point of the domain curve, and one of the irreducible components meeting at $p$ should be mapped to the base $\PO$, which by the choice of $f$ is just $C_{b}$. This means that $p$ is in the intersection of $C^{\bullet}_{1}$ and $C_{b}$, but the genus of the source curve is zero, so there is just one such point. The same argument shows that in each bubble contains the image of exactly one irreducible component of $C$.

So $C^{0}_{1}$ and $C^{\infty}_{1}$ are degree d curves with full contact with a hyperplane (the exceptional divisor of the blow up) in $\PP^3$ at a single point, and since $f$ is $\CS$ fixed they are also fixed under the induced action.

 From the description of fixed points, we see that the contact point of the image and the exceptional divisor of the blow up of $\PP^3$, for the generic point of a fixed component with non-zero contribution, can be either [0;1;0] or [0;0;1], which means that on the next bubble the image will hit the associated divisor at [0;1;0;0] or [0;0;1;0]. If we consider the fixed maps with these contact points then the same argument as above shows that they can be presented in one of the following forms (see Table~\ref{fixed points}).

\begin{table}[!ht]
\begin{tabular}{|c|l|c|}

\hline
&&\\
Contact point  &  $f: \PO \To \PP^3$ & weights of the $\CS$ action \\
 & & on $V^4$ \\
\hline
&&\\
\multirow{7}{*}{[1;0;0;0]} & $[x;y]\mapsto[x^d;l x^{h}y^{d-h};m x^{k}y^{d-k};y^d]$ &  \\
 & \hspace{50mm}$d+h=2k$ & $\{\frac{d}{d-k}\alpha,\frac{h}{d-k}\alpha,\frac{k}{d-k}\alpha,0 \}$ \\
 & \hspace{50mm}$h,k \leq d-1$ & \\
 & $[x;y]\mapsto[x^d; x^{h}y^{d-h}; 0;y^d]$ & $\{\frac{2d}{d-h}\alpha,\frac{2h}{d-h}\alpha,\frac{h+d}{d-h}\alpha,0 \}$ \\
 & \hspace{50mm}$h \leq d-1 $ & \\
 & $[x;y]\mapsto[x^d; 0; x^{k}y^{d-k};y^d]$ & $\{\frac{d}{d-k}\alpha,\frac{2k-d}{d-k}\alpha,\frac{k}{d-k}\alpha,0 \}$ \\
 & \hspace{50mm}$k \leq d-1 $& \\
  \hline
  &&\\
\multirow{7}{*}{[0;1;0;0]} & $[x;y]\mapsto[l x^{h}y^{d-h};x^d;m x^{k}y^{d-k};y^d]$ &  \\
 & \hspace{50mm}$d+h=2k$ & $\{\frac{-d}{d-k}\alpha,\frac{-h}{d-k}\alpha,\frac{-k}{d-k}\alpha,0 \}$\\
 & \hspace{50mm}$h,k \leq d-1$ & \\
 & $[x;y]\mapsto[ x^{h}y^{d-h};x^d; 0;y^d]$ & $\{\frac{-2d}{d-h}\alpha,\frac{-2h}{d-h}\alpha,\frac{-h-d}{d-h}\alpha,0 \}$ \\
 & \hspace{50mm}$h \leq d-1 $ & \\
 & $[x;y]\mapsto[ 0;x^d; x^{k}y^{d-k};y^d]$ & $\{\frac{-d}{d-k}\alpha,\frac{d-2k}{d-k}\alpha,\frac{-k}{d-k}\alpha,0 \}$ \\
 & \hspace{50mm}$k \leq d-1 $ & \\
  \hline
  &&\\
 \multirow{4}{*}{[0;0;1;0]} & $[x;y]\mapsto[ x^{h}y^{d-h}; 0;x^d;y^d]$ & $\{\frac{-d}{d-h}\alpha,\frac{-h}{d-h}\alpha,\frac{h-2d}{d-h}\alpha,0 \}$ \\
 & \hspace{50mm}$h \leq d-1 $ & \\
 & $[x;y]\mapsto[ 0; x^{k}y^{d-k};x^d;y^d]$ & $\{\frac{d}{d-k}\alpha,\frac{2d-k}{d-k}\alpha,\frac{k}{d-k}\alpha,0 \}$ \\
 & \hspace{50mm}$k \leq d-1 $ & \\
  \hline
\end{tabular}
\vspace{3mm}\caption{Fixed points}
 \label{fixed points}
 \end{table}
\vspace{3mm}

Now we compute the induced $\CS$ action on the source curve of the fixed points for which the map $f$ is of the form:

\begin{center}
     $\begin{array} {ccccl}
       f & : & \PO & \To & \PP^3 \\
         & & [x;y] & \mapsto & [ x^d;l x^{h}y^{d-h};m x^{k}y^{d-k};y^d]
 \end{array}$
 \end{center}\vspace{3mm}

 \noindent the action is given by :

 \begin{center}
     $ t.[ x^d;l x^{h}y^{d-h};m x^{k}y^{d-k};y^d]= [t^{-1} x^d;t l x^{h}y^{d-h};m x^{k}y^{d-k};y^d]$
 \end{center}\vspace{3mm}

 \noindent but we if we dilate with ratio $\w ^{k}$ , with $\w =t^{\frac{-1}{d-k}}$,(considered as an automorphism of the $\A^3\subset\PP^3$, which is an automorphism of the target as a FM-configuration) then the action can be written as:

 \begin{center}
     $ t.[ x^d;l x^{h}y^{d-h};m x^{k}y^{d-k};y^d]= [\w^{d} x^d;\w^{h} l x^{h}y^{d-h};\w^{k} m x^{k}y^{d-k};y^d]$
 \end{center}\vspace{3mm}

This means that for such $f$ the induced action on the source curve is given by $t.[x;y]=[\w x;y]$, and also the weights of the action on the 4-dimensional vector space $V^4$ (such that the bubble $\PP^3$ is isomorphic to $\mathbb{P}V^4$) are $\{\frac{d}{d-k}\alpha,\frac{h}{d-k}\alpha,\frac{k}{d-k}\alpha,0\}$.

Note that despite the fact that we considered the bubble over $0$ the above description of the fixed maps, works for any bubble up to permutation of coordinates. The weights of the action on $V^4$(defined as above) for all fixed point maps are similarly computed in Table~\ref{fixed points}.

%%%%%%%%%%%%%%%%%%%%%%%%%%%%%%%%%%%%%%%%%%%%%%%%%%%%%%%%%%%%%%%%%%%%%%%%%%%%%%%%%%%%%%%%%%%%%%%%%%%%%%%%%%%%%%%%%%%%%%%%%%%%%%%%%%%%%%%%%%%%%%%%%%%%
\section{Computation}

\subsection{Contribution of Each Fixed Locus}

Note that since we want to compute the euler class of the normal bundle of the fixed point loci with non zero contribution to integral. By the above discussion we have to consider maps with only one component being mapped to the $\PO \subset X$ and ramified at 0 and $\infty$.

To compute the euler class of the normal bundle, using $\ast$ we will break both source and target into irreducible components. By the result of previous section we have to consider three types of maps:
\begin{itemize}
\item Maps to the blow up of $X$ at 0 and $\infty$.
\item Maps to the middle bubbles which are isomorphic to blow up of $\PP^3$ at one point.
\item Maps to the end bubbles which are isomorphic to $\PP^3$.
\end{itemize}

\begin{rem}\textbf{Local Deformation/Obstruction:}

The term coming from the local obstruction, which is given by $\EExt^{1}(\Omega^{\dag}_{W},\mathcal{O}_{W})$ , have the same contribution as the term for smoothing the singularities of the target along the divisors $\left\{D^{\bullet}_{i}\right\}$. For an explicit computation see \cite[section 2.8]{GV}. Since there is only one node being mapped to each $D^{\bullet}_{i}$ their contributions will cancel out.

Note that at each node of the source curve we have a term coming from smoothing that node, which is a one dimensional space isomorphic to $T_p{C^\bullet_{i}}\otimes T_p{C^{\bullet}_{i-1}}$.
\end{rem}

\begin{enumerate}
  \item \textbf{Maps to the blow up of $X$ at 0 and $\infty$ :}

  The deformations of the target are given by moving the two bubbles, which give us two three dimensional deformation spaces which are isomorphic to $T_{0}(X)$ and $T_{\infty}(X)$. The weights of the $\CS$ action on each of these spaces are $\{\alpha,0,-\alpha\}$. There is a one dimensional automorphism group of automorphism of the source(given by dilation, which fixes both nodes) and the weight of the $\CS$ action on this space is $0$.

  To compute the contribution of the deformation/obstruction of the map, we use the long exact sequence for logarithmic differentials. In the K-theory we have: \vspace{3mm}

$\begin{array}{l}
(H^{0}-H^{1})(C_{b},f^{*}T_{B_{0}}(-logD^{0}_{0}-logD^{\infty}_{1})) = \\
\hspace{20mm}= (H^{0}-H^{1})(C_{b},f^{*}T_{B_{0}}) - H^{0}(C_{b},f^{*}N_{D^{0}_{0}}B_{0}) -  H^{0}(C_{b},f^{*}N_{D^{\infty}_{0}}B_{0})
\end{array}$ \vspace{3mm}

To compute the weights of the action on $(H^{0}-H^{1})(C_{b},f^{*}T_{B_{0}})$ using the blow up exact sequence we get:\vspace{3mm}

$\begin{array}{l}
(H^{0}-H^{1})(C_{b},f^{*}T_{B_{0}})=\\
\hspace{20mm}=(H^{0}-H^{1})(C_{b},f^{*}T_{X})  - H^{0}(C_{b},f^{*}j_{*}\mathcal{Q}_{0}) - H^{0}(C_{b},f^{*}j_{*}\mathcal{Q}_{\infty}) \\
\hspace{20mm}=(H^{0}-H^{1})(C_{b},f^{*}T_{X}) \\
\hspace{22mm}- H^{0}(C_{b},f^{*}\mathcal{O}_{D_0^0})\otimes V_0^3 + H^{0}(C_{b},f^{*}N_{D^{0}_{0}}B_{0})\\
\hspace{22mm} - H^{0}(C_{b},f^{*}\mathcal{O}_{D_0^{\infty}})\otimes V_{\infty}^3 +  H^{0}(C_{b},f^{*}N_{D^{\infty}_{0}}B_{0})
\end{array}$ \vspace{3mm}

Thus we have:\vspace{3mm}

$\begin{array}{l}
(H^{0}-H^{1})(C_{b},f^{*}T_{B_{0}}(-logD^{0}_{0}-logD^{\infty}_{0})) = \\
\hspace{15mm}= (H^{0}-H^{1})(C_{b},f^{*}T_{B_{0}}) - H^{0}(C_{b},f^{*}N_{D^{0}_{0}}B_{0}) -  H^{0}(C_{b},f^{*}N_{D^{\infty}_{0}}B_{0}) \\
\hspace{15mm}= (H^{0}-H^{1})(C_{b},f^{*}T_{X}) - H^{0}(C_{b},f^{*}\mathcal{O}_{D_0^0})\otimes V_0^3 - H^{0}(C_{b},f^{*}\mathcal{O}_{D_0^{\infty}})\otimes V_{\infty}^3
\end{array}$ \vspace{3mm}

Note that $f^{*}\mathcal{O}_{D_0^{\bullet}}$ is isomorphic to $ \mathcal{O}_{C_{b}}/ \mathfrak{m}_p^d$ where $p$ is the node of $C_{b}$ above the point $0\in \PO$. We have the following exact sequence:

$$0 \To  (\mathfrak{m}_p/\mathfrak{m}_p^2)^{\otimes k-1} \To \mathcal{O}_{C_{b}}/ \mathfrak{m}_p^k \To \mathcal{O}_{C_{b}}/ \mathfrak{m}_p^{k-1} \To 0$$

\noindent so the weights of the action on $H^{0}(C_{b},f^{*}\mathcal{O}_{D_0^{\bullet}})$ are given by $\{\frac{i}{d}\alpha\}_{i=0}^{d-1}$.

Since $T_X=T_{\PP^1}\oplus\mathcal{O}(-1)\oplus \mathcal{O}(-1)$, the contribution of $(H^{0}-H^{1})(C_{b},f^{*}T_{X})$ to the normal bundle, after cancellation with a weight zero contribution of the automorphism of the source, is $-\alpha^2$, since after the cancellation of the weight zero term the contribution of $(H^{0}-H^{1})(C_{b},f^{*}T_{\PP^1})$ is $$\alpha^{2d}\frac{(-1)^d (d!)^2}{d^{2d}},$$ and the contribution of $(H^{0}-H^{1})(C_{b},f^{*}\mathcal{O}(-1)\oplus \mathcal{O}(-1))$ is $$\alpha^{2-2d}\frac{(-1)^{d-1}d^{2d-2}}{(d-1)!^2}.$$

The contribution of $- H^{0}(C_{b},f^{*}\mathcal{O}_{D_0^0})\otimes V_0^3 - H^{0}(C_{b},f^{*}\mathcal{O}_{D_0^{\infty}})\otimes V_{\infty}^3$ to the normal bundle (after cancellation with the two weight zero terms coming from moving the bubbles) is $$(-1)^{3d-1} \alpha^{2-6d} (\frac{d^{3d-1}}{d!(2d-1)!})^2.$$

Moving the bubbles give $\alpha^{4}$,$-T_{D}|_{f(N)}$ terms give $\frac{1}{4} \alpha^{-4}$. Hence the weight of the action on the normal bundle is 
 $$(-1)^{3d-1}\alpha^{4-6d} (\frac{d^{3d}}{d!(2d)!})^2.$$

Note that there is also a factor of $\frac{1}{d}$ for the automorphisms of the source and since the contribution to the localization formula is $\frac{1}{e(N)}$ so the total contribution of the base component is
$$\frac{(-1)^{3d-1}}{d}\alpha^{6d-4} (\frac{d!(2d)!}{d^{3d}})^2 .$$

\item \textbf{Maps to the ruled bubbles :}

By the description of fixed maps from previous section, we know that each bubble in the middle has exactly one point (say $p$) blown up, so after fixing $p$ to be the origin in $\A^3$,the only automorphism/deformation terms that we have to consider are dilation automorphisms, which is a one dimensional group.

The contribution from deformation/obstruction of $f$ can be computed by the weights of the $\CS$ action on $(H^0-H^1)(C_i,f^{*}T_{B_i}(-logD_{i}-logD_{i-1}))$. Using the exact sequence for logarithmic tangent bundle and the blow up exact sequence we have:\vspace{3mm}

$\begin{array}{l}
(H^0-H^1)(C_i,f^{*}T_{B_i}(-logD_{i}-logD_{i-1})) \vspace{2mm} \\
\hspace{30mm} = (H^0-H^1)(C_i,f^{*}T_{B_i}(-logD_{i-1}))- H^{0}(C_{i},f^{*}N_{D_{i}}B_{i}) \vspace{2mm} \\
\hspace{30mm} = (H^0-H^1)(C_i,f^{*}T_{\PP^3}(-logD_{i-1})) +  H^{0}(C_{i},f^{*}N_{D_{i}}B_{i}) \\
\hspace{54mm} - V^3 \otimes H^{0}(C_{i},f^{*}\mathcal{O}_{D_{i}}) - H^{0}(C_{i},f^{*}N_{D_{i}}B_{i}) \vspace{2mm} \\
\hspace{30mm} = W^3 \otimes H^0(C_i,f^{*}\mathcal{O}_{\PP^3}(1))- V^3 \otimes H^{0}(C_{i},f^{*}\mathcal{O}_{D_{i}})
\end{array}$ \vspace{3mm}

We have to consider several cases, according to the contact point with the distinguished divisor in $\PP^3$ and contact point with the exceptional divisor of the blowup as in Table~\ref{fixed points}.

If the contact point is $[1:0:0:0]$, using Table~\ref{fixed points} we see that in one case there is a 1-dimensional fixed point locus given by :

 $$\begin{array} {ccccll}
       f & : & \PO & \To & \PP^3 & \\
         & & [x;y] & \mapsto & [ x^d;l x^{h}y^{d-h};m x^{k}y^{d-k};y^d] & \text{with } d+h=2k
 \end{array}$$

 For these fixed points there are three weight zero terms coming from the deformation of the map, and two of them cancel out with similar terms in the automorphisms of the source and target, and the remaining weight zero term is in the fixed part of the deformation space and in fact is just the tangent to the fixed locus. Hence the contribution of this component should be computed by integrating the restriction of the obstruction bundle over the fundamental class of this component.

 \textbf{Obstruction bundle:} The obstruction bundle has a trivial summand coming from predeformability condition, which is isomorphic to $T_D$. The fiber of the other part of the obstruction bundle over a given map is given by $V^3 \otimes H^{0}(C_{i},f^{*}\mathcal{O}_{D_{i}})$.

 Before we compute the first Chern class of the obstruction bundle over this 1-dimensional component, we digress and describe the maps in this locus.

 Given a map of the form $[ x^d;l x^{h}y^{d-h};m x^{k}y^{d-k};y^d]$ up to an automorphism of the source and the target it is equivalent to $[v w^d x^d;v w^h l x^{h}y^{d-h};v w ^k m x^{k}y^{d-k};y^d]$.

 If $l \neq 0$ pick $w=l^{\frac{1}{d-h}}$ and $v=l^{\frac{-d}{d-h}}$ we get $[x^d; x^{h}y^{d-h}; m l^{\frac{-1}{2}} x^{k}y^{d-k};y^d]$, and if $m \neq 0$ pick $w=m^{\frac{1}{d-k}}$ and $v=m^{\frac{-d}{d-k}}$ we get $[x^d;m^{-2} l x^{h}y^{d-h}; x^{k}y^{d-k};y^d]$.

 Now it is easy to see that if $m$ goes to zero the limiting map is given by $[x^d;x^{h}y^{d-h}; 0 ;y^d]$, but if $l$ goes to zero then the source and target will bubble up and the limiting map has two component instead of one. The map for the component in the first bubble is given by $[x^d;0; x^{k}y^{d-k};y^d]$ and the map in the second bubble is given by: $[0; x^{h}y^{k-h}; x^{k};y^k]$.

%////////////////ADD AN IMAGE HERE///////////////////////

Let $\mathfrak{Ob}$ be the bundle with fiber $V^3 \otimes H^{0}(C,f^{*}\mathcal{O}_{D_s})$ over a given fixed map $f_s$ in this family, where $D_s$ is the distinguished divisor for the given map in the family (so the image of $f_s$ hits $D_s$ with multiplicity $h$), and $V^3$ is the tangent space of the point $[0:0:0:1]$ in $\PP^3$.

Note that the fiber of $\mathfrak{Ob}$ over a given curve is isomorphic to $  V^3\otimes\mathcal{O}_{C_s}/ \mathfrak{m}_{p_s}^h$ where $p_s$ is the contact point of the image of $C_s$ and $D_s$. Let $\mathfrak{Ob}_{i}$ be the bundle with fiber $ \mathcal{O}_{C_s}/ \mathfrak{m}_{p_s}^i$ over a given curve, in particular  $\mathfrak{Ob}=V^3\otimes\mathfrak{Ob}_{h}$.

We have the following exact sequence:

$$0 \To  (\mathfrak{m}_{p_s}/\mathfrak{m}_{p_s}^2)^{\otimes i-1} \To \mathcal{O}_{C_s}/ \mathfrak{m}_{p_s}^i \To \mathcal{O}_{C_s}/ \mathfrak{m}_{p_s}^{i-1} \To 0$$

The above exact sequence shows:

$$0 \To  \mathbb{L}^{\otimes i-1} \To \mathfrak{Ob}_{i} \To \mathfrak{Ob}_{i-1} \To 0$$

\noindent where $\mathbb{L}$ is the the bundle whose fiber over a given curve $C_s$ is the cotangent space of the point $p_s \in C_s$. We denote the first Chern class of $\mathbb{L}$ by $\psi$. The weight of the $\CS$ action on the fibers of $\mathbb{L}$ is $\frac{-\alpha}{d-k}$, so the equivariant first Chern class of $\mathbb{L}$ is $\frac{-\alpha}{d-k}+\psi$. The weights of the action on $V^3$ are $\{\frac{h}{d-k},\frac{k}{d-k},\frac{d}{d-k}\}$. Hence the equivariant total Chern class of $\mathfrak{Ob}$ is:

$$\displaystyle \prod_{i=0}^{h-1} (1+\frac{(h-i)\alpha}{d-k}+i\psi)(1+\frac{(k-i)\alpha}{d-k}+i\psi)(1+\frac{(d-i)\alpha}{d-k}+i\psi)$$

\noindent Note that $\psi^2=0$ so it can be written as :

$$\displaystyle \left(\frac{1}{d-k}\right)^{3h} \frac{h!k!d!}{(d-h)!(k-h)!} \left[\alpha^{3h}+(d-k) \psi \alpha^{3h-1}\left( \sum_{i=0}^{h-1} \frac{i}{h-i}+\frac{i}{k-i}+\frac{i}{d-i}  \right)\right]$$

Note that the weights of the action on $W^3$ are $\{\frac{h}{n-k}\alpha,\frac{k}{d-k}\alpha,\frac{d}{d-k}\alpha\}$, and the weights of the action on $H^0(C_i,f^{*}\mathcal{O}_{\PP^3}(1))$ are $\{\frac{-i}{d}\alpha\}_{i=0}^{d}$, so the contribution of 
$W^3 \otimes H^0(C_i,f^{*}\mathcal{O}_{\PP^3}(1))- V^3 \otimes H^{0}(C_{i},f^{*}\mathcal{O}_{D_{i}})$ is equal to 
$$(-1)^{h+k}[\frac{1}{(d-h)!(d-k)!}]^2(\frac{\alpha}{d-k})^{3h-3d-1} \psi  \left( \sum_{i=0}^{h-1} \frac{i}{h-i}+\frac{i}{k-i}+\frac{i}{d-i}  \right)$$

Similarly we can compute the contribution of other fixed point maps, the values are shown in Table~\ref{IntegralTable 1}.

\begin{table}[!h]
\begin{tabular}{|l|c|c|}

\hline
& &\\
  $f: \PO \To \PP^3$ & The contribution of  & $T_D$, \\
  & \tiny{$W^3 \otimes H^0(C_i,f^{*}\mathcal{O}_{\PP^3}(1))- V^3 \otimes H^{0}(C_{i},f^{*}\mathcal{O}_{D_{i}})$} &  automorphisms \\
  & & \\
\hline
& & \\
  $[x;y]\mapsto[x^d;l x^{h}y^{d-h};m x^{k}y^{d-k};y^d]$ &  & \\
  \hspace{30mm}$d+h=2k$ & $(-1)^{h+k}[\frac{1}{(d-h)!(d-k)!}]^2(\frac{\alpha}{d-k})^{3h-3d-1} \psi \cdot$ & $2\alpha	^2. \frac{1}{d-k}$ \\
  \hspace{30mm}$h,k \leq d-1$ & $\left( \sum_{i=0}^{h-1} \frac{i}{h-i}+\frac{i}{k-i}+\frac{i}{d-i}  \right)$& \\
  & &\\
 $[x;y]\mapsto[x^d; x^{h}y^{d-h}; 0;y^d]$ & $(-1)^{d+\lceil\frac{d+h}{2}\rceil}2^{d-h}[\frac{1}{(d-h)!(d-h)!!}]^2(\frac{2\alpha}{d-h})^{3h-3d-1}$ &$2\alpha^2 \frac{1}{d-h}$\\
 \hspace{30mm}$h \leq d-1 $ &  &\\
 \hspace{30mm}$h\neq d(mod 2)$ & & \\
  & & \\
  $[x;y]\mapsto[x^d; 0; x^{k}y^{d-k};y^d]$ & $\frac{-1}{d-k}\frac{1}{(d-k)!(2d-2k)!}(\frac{\alpha}{d-k})^{3k-3d-1}$ &$-\alpha^2\frac{1}{d-k}$\\
\hspace{30mm}$k \leq d-1 $& &\\

  \hline
 & &\\
 $[x;y]\mapsto[l x^{h}y^{d-h};x^d;m x^{k}y^{d-k};y^d]$ &  &\\
 \hspace{30mm}$d+h=2k$ & $(-1)^{h+k+1}[\frac{1}{(d-h)!(d-k)!}]^2(\frac{\alpha}{d-k})^{3h-3d-1}$ &$2\alpha^2. \frac{1}{d-k}$\\
  \hspace{30mm}$h,k \leq d-1$ & $\left( \sum_{i=0}^{h-1} \frac{i}{h-i}+\frac{i}{k-i}+\frac{i}{d-i}  \right) $&\\
  & &\\
  $[x;y]\mapsto[ x^{h}y^{d-h};x^d; 0;y^d]$ & $(-1)^{d+\lceil\frac{d+h}{2}\rceil}2^{d-h}[\frac{1}{(d-h)!(d-h)!!}]^2(\frac{2\alpha}{d-h})^{3h-3d-1}$ &$2\alpha^2 \frac{1}{d-h}$\\
  \hspace{30mm}$h \leq d-1 $ & &\\
\hspace{30mm}$h\neq d(mod 2)$ & &\\
  & &\\
  $[x;y]\mapsto[ 0;x^d; x^{k}y^{d-k};y^d]$ & $\frac{(-1)^{d+k}}{d-k}\frac{1}{(d-k)!(2d-2k)!}(\frac{\alpha}{d-k})^{3k-3d-1}$ & $-\alpha^2\frac{1}{d-k}$\\
 \hspace{30mm}$k \leq d-1 $ & & \\

  \hline
&& \\
 $[x;y]\mapsto[ x^{h}y^{d-h}; 0;x^d;y^d]$ & $\frac{-1}{d-h}\frac{1}{(d-h)!(2d-2h)!}(\frac{\alpha}{d-h})^{3h-3d-1}$& $2\alpha^2 \frac{1}{d-h}$\\
 \hspace{30mm}$h \leq d-1 $ & &\\
  & &\\
   $[x;y]\mapsto[ 0; x^{k}y^{d-k};x^d;y^d]$ & $\frac{(-1)^{d-k}}{d-k}\frac{1}{(d-k)!(2d-2k)!}(\frac{\alpha}{d-k})^{3k-3d-1}$  &$2\alpha^2\frac{1}{d-k}$\\
 \hspace{30mm}$k \leq d-1 $ & &\\

  \hline
\end{tabular}
\vspace{3mm}\caption{Maps to the ruled bubbles}
 \label{IntegralTable 1}
 \end{table}
\vspace{3mm}

\begin{rem}
One should note that any map of the form  $[x;y]\mapsto[x^d; x^{h}y^{d-h}; 0;y^d]$ with $h=d(mod 2)$ is an element of the family  $[x;y]\mapsto[x^d;l x^{h}y^{d-h};m x^{k}y^{d-k};y^d]$. So we do not have to consider their contribution separately, and that is why in the Table~\ref{IntegralTable 1} we consider the case $d\neq h (mod2)$.
\end{rem}

 \textbf{Integral of $\psi$:}
In order to compute the integral of $\psi$ over the one dimensional fixed point locus, we use the localization again. A generic member of the fixed point loci is given by $[x;y]\mapsto[x^d;l x^{h}y^{d-h};m x^{k}y^{d-k};y^d]$, and we consider the $\CS$-action on the target given by $t.[x_0:x_1:x_2:x_3]\mapsto[t^{-1} x_0:t x_1:t x_2:x_3]$. This action has only two fixed points:
\begin{itemize}
\item The map given by $[x;y]\mapsto[x^d; x^{h}y^{d-h}; 0;y^d]$.
\item The map whose source curve has two components, and maps on the two components are given by:  $[x;y]\mapsto[x^d; 0; x^{k}y^{d-k};y^d]$ and $[x;y]\mapsto[ 0; x^{k}y^{d-k};x^d;y^d]$.
\end{itemize}

   The weight of the action on the normal bundle of the first fixed point is $2\beta$ (where $\beta$ is the weight of the dual of the standard representation), and the weight of the action on the fiber of $\psi$ over this point is $\frac{-2}{d-h}\beta$. The weight of the action on the fiber of $\psi$ over the second fixed point is zero so the contribution of the second fixed point to the integral is zero. Hence the integral of the $\psi$ over this locus is $\frac{-1}{d-h}$.

  \item \textbf{Maps to the $\PP^3$ :}

  Similarly we have:

$$(H^0-H^1)(C_i,f^{*}T_{B_i}(-logD_{i-1}))= W^3 \otimes H^0(C_i,f^{*}\mathcal{O}_{\PP^3}(1))$$

\noindent where $W^3$ is the tangent space of the point $[0:0:0:1]$ in $\PP^3$. One should note that similar to the case of maps to the ruled bubbles, in some cases we have a one dimensional fixed locus. In this case the group of isomorphisms of the source that fix the the node, is two dimensional which decomposes into two eigenspaces. One of which is just the tangent space over the point $[0:1]$, which is the dual of the $\psi$ which we considered in the last case. The other eigen space is trivial as a bundle and has $\CS$-weight equal to zero. There is one more minor change which comes from deformations of the target. In this case since the target is $\PP^3$ we have a four  dimensional group of automorphisms of the target that has a weight zero term for the dilation and the other three weights are the weight of the action on tangent space of the point $[0;0;0;1]$. Similar to the cases with ruled bubble targets, in Table~\ref{IntegralTable 2} we list all the possible maps and their contributions. 

\begin{table}[!ht]
\begin{tabular}{|l|c|}

\hline
  $f: \PO \To \PP^3$ & The contribution of  \\
  & \tiny{$W^3 \otimes H^0(C_i,f^{*}\mathcal{O}_{\PP^3}(1))$} \\
\hline
  $[x;y]\mapsto[x^d;l xy^{d-1};m x^{k}y^{d-k};y^d]$ &  \\
  \hspace{30mm}$d+1=2k$ & $(-1)^{k+1}[\frac{1}{(d-1)!(d-k)!}]^2(\frac{\alpha}{d-k})^{3-3d} \psi^{\vee} \cdot$\\
  & \\
 $[x;y]\mapsto[x^d; xy^{d-1}; 0;y^d]$ & $(-1)^{d+\lceil\frac{d+1}{2}\rceil}2^{d}[\frac{1}{(d-1)!(d-1)!!}]^2(\frac{2\alpha}{d-1})^{3-3d}$ \\
 \hspace{30mm}$d \neq 1(mod 2)$ & \\
  & \\
  $[x;y]\mapsto[x^2; 0; xy;y^2]$ & $\frac{(-1)}{2}\alpha^{-3}$ \\

  \hline
 $[x;y]\mapsto[l x y^{d-1};x^d;m x^{k}y^{d-k};y^d]$ &  \\
 
 \hspace{30mm}$d+1=2k$ & $(-1)^{k}[\frac{1}{(d-1)!(d-k)!}]^2(\frac{\alpha}{d-k})^{3-3d} \psi^{\vee} \cdot$\\

  & \\
  $[x;y]\mapsto[ xy^{d-1};x^d; 0;y^d]$ & $(-1)^{d+\lceil\frac{d+1}{2}\rceil}2^{d}[\frac{1}{(d-1)!(d-1)!!}]^2(\frac{2\alpha}{d-1})^{3-3d}$ \\
\hspace{30mm}$d\neq 1(mod 2)$ & \\
  & \\
  $[x;y]\mapsto[ 0;x^2; xy;y^2]$ & $\frac{(-1)}{2}\alpha^{-3}$  \\

  \hline
 $[x;y]\mapsto[ xy^{d-1}; 0;x^d;y^d]$ & $\frac{-1}{d-1}\frac{1}{(d-1)!(2d-2)!}(\frac{\alpha}{d-1})^{3-3d}$ \\
  & \\
   $[x;y]\mapsto[ 0; xy^{d-1};x^d;y^d]$ & $\frac{(-1)^{d-1}}{d-1}\frac{1}{(d-1)!(2d-2)!}(\frac{\alpha}{d-1})^{3-3d}$  \\
 \hspace{30mm}$k \leq d-1 $ & \\

  \hline
\end{tabular}
\vspace{3mm}\caption{Maps to the $\PP^3$}
 \label{IntegralTable 2}
 \end{table}
\vspace{3mm}
\end{enumerate}

\subsection{Explicit Computations}

\subsubsection{The degree 2 case}

The contribution of the component that is mapped to the base is 
$$\frac{(-1)^{3d-1}}{d}\alpha^{6d-4} (\frac{d!(2d)!}{d^{3d}})^2 =-\alpha^{8}\frac{9}{32}$$

\emph{Smoothing the node over $0$}: If the map is given by $[x^2; 0; xy;y^2]$ smoothing the node gives $-\frac{2}{3\alpha}$ and the normal bundle contribution is $\frac{-1}{2}\alpha^{-3}$. If the map is given by $[x^2; xy ; 0 ; y^2]$ the smoothing contribution is $-\frac{2}{5\alpha}$ and normal bundle contribution is $\frac{1}{2}\alpha^{-3}$. So in the bubble above $0$ the contribution is $\frac{2}{15}\alpha^{-4}$. The contribution of the bubble above $\infty$ is similar and we have to change the $\alpha$ to $-\alpha$. So the integral for the double cover is $-\frac{1}{2^3 5^2}$.  

\subsubsection{The general case} In general we have to consider all the possible ways that the fixed maps in the bubbles could have. In the first bubble above zero the contact point with the distinguished divisor is $[1;0;0;0]$, and by Table~\ref{fixed points} the contact point of the map in the next bubble with the distinguished divisor will be either $[0;1;0;0]$ or $[0;0;1;0]$. For each of these cases we have to consider the possible maps, and since the degree of map decreases from one bubble to the next, this process will stop. By Table~\ref{IntegralTable 1} and Table~\ref{IntegralTable 2} we can compute the contribution of each of these components, and by multiplying them we get the total contribution of that fixed locus. We have a finite number of such configurations, and by adding their contributions we get the value of the integral. 

The author has written a computer program for this computation, in the following table we have the values for $d<10$.
\begin{table}[!h]
\begin{tabular}{|l|c|}\hline
$d$ & contribution\\
\hline
2 & $-\frac{1}{2^3 5^2}$\\
\hline
3 & $-\frac{5^2 43^2}{3^{13}7^{ 2}}$\\
\hline
4 & $-\frac{3^6 7^2 233^2}{2^{44} 11^2}$\\
\hline
5 & $-\frac{3^2  {190562513^2}}{2^4 5^{25} 11^2 13^2}$\\
\hline  
6 & $-\frac{ {11^ 2} 37^2 107^2 1699^2 22531^2}{2^{37} 3^{29} 13^2 17^2}$\\
\hline
7 & $-\frac{{3^6}{11^2}{13^2}{61^2}{73^2}{122439620123^2}}{{2^12}{5^2}{7^37}{17^2}{19^2}}$\\
\hline
8 & $-\frac{{11740987^2}{49789008475889939^2}}{{2^145}{3^2}{17^2}{19^2}{23^2}}$\\
\hline
9 & $-\frac{{11^2}{13^2}{17^2}{3373^2}{38721049^2}{1842598673213^2}}{{2^12}{3^98}{5^4}{19^2}{23^2}}$\\
\hline
\end{tabular}
\end{table}

% ----------------------------------------------------------------
%\bibliographystyle{amsplain}
\bibliography{}

\end{document}